\documentclass[12pt,a4paper]{article}
\usepackage{amssymb}
\usepackage{amsmath}
\usepackage{amsthm}
\usepackage{array}
\usepackage{rotating}
\usepackage{verbatim}
\usepackage[all]{xy}

\usepackage{mathtools}

\theoremstyle{plain}

\mathtoolsset{showonlyrefs=true}

\newtheorem{theorem}{Theorem}[section]
\newtheorem{corollary}[theorem]{Corollary}
\newtheorem{lemma}[theorem]{Lemma}
\newtheorem{conjecture}[theorem]{Conjecture}
\newtheorem{proposition}[theorem]{Proposition}

\newtheorem{definition}[theorem]{Definition}
\newtheorem{remark}[theorem]{Remark}

\newtheoremstyle{indenteddefinition}{\topsep}{\topsep}{\addtolength{\leftskip}{2.0em}}{-0em}{\bfseries}{.}{
}{}
\theoremstyle{indenteddefinition}

\DeclareMathOperator\Ad{Ad}

\DeclareMathOperator\Aut{Aut}

\DeclareMathOperator\Hom{Hom}
\DeclareMathOperator\Gal{Gal}

\DeclareMathOperator\Int{Int}

\DeclareMathOperator\Irr{Irr}

\DeclareMathOperator\im{im}

\DeclareMathOperator{\Ker}{Ker}

\DeclareMathOperator{\SL}{SL}

\makeatletter

\@addtoreset{equation}{section}
\makeatother

\newcommand{\Ga}{\mathbb{G}_{\mathrm{a}}}

\newcommand{\sep}{\mathrm{sep}}

\begin{document}
\title{Langlands parameters for reductive groups over finite fields}
\author{Naoki Imai and David A. Vogan, Jr.} 

\date{}

\maketitle

\begin{abstract}
We define Langlands parameters for connected reductive groups over finite fields and formulate the Langlands correspondence for finite fields using these parameters. 
\end{abstract}

\section{Introduction}\label{sec:intro}
\setcounter{equation}{0}

The goal of this paper is try to formulate a parametrization of
representations of a finite Chevalley group $G({\mathbb F}_q)$ over $\overline{\mathbb{Q}}_{\ell}$ in terms of the Langlands dual group over $\overline{\mathbb{Q}}_{\ell}$ of $G$. One motivation is to relate
this parametrization for $G({\mathbb F}_q)$ to the (still
conjectural!) Langlands parametrization of irreducible representations
of groups over local fields. The reason for using the 
dual group over $\overline{\mathbb{Q}}_{\ell}$ is that Langlands' philosophy suggests that representations of a reductive group $G$ on $k=\overline k$ vector spaces ought to be related to (maps of a Weil group to) a $k$-dual group. 
Such a parametrization was given for representations of $GL_n({\mathbb
F}_q)$ by Ian Macdonald in \cite{MacZetafingen}.

For more general Chevalley groups, parametrizations
involving the dual group over $\overline{\mathbb F}_q$ may be found in
\cite{DeLuRep}, \cite{LusChred}, and \cite{LusRepdis}. A formulation using the
complex dual group is stated in \cite{LusChredICM}, but even this formulation
is not so well connected to Langlands parameters. The reason is that
the Deligne--Lusztig and Lusztig formulations are stated in terms of a
single {\em element} of the dual group. In the most fundamental and
simplest example where $G({\mathbb F}_q) = {\mathbb F}_q^\times$, a
representation is simply a character of the multiplicative group
${\mathbb F}_q^\times$. The dual group over $\overline{\mathbb F}_q$ is
$\overline{\mathbb F}_q^\times$. The special dual group elements that
Deligne and Lusztig consider in \cite{DeLuRep} and \cite{LusRepdis} are elements of
order $q-1$; that is, elements of ${\mathbb F}_q^\times$. They are
therefore seeking to parametrize {\em characters} of ${\mathbb
  F}_q^\times$ by {\em elements} of ${\mathbb F}_q^\times$. Because
the multiplicative group of a finite field is cyclic, such a
parametrization is possible, but it is never natural. The choices
required appear in \cite[(5.0.1)--(5.0.2)]{DeLuRep}: isomorphisms
\begin{equation}\begin{aligned}\label{e:choices}
\overline{\mathbb F}_q^\times &\buildrel{\sim}\over\longrightarrow ({\mathbb
Q}/{\mathbb Z})_{p'}\\
\text{roots of $1$, order prime to $p$ in $\overline{\mathbb{Q}}_{\ell}^\times$} &\buildrel{\sim}\over\longrightarrow ({\mathbb
Q}/{\mathbb Z})_{p'}.\qquad\qquad \end{aligned}
\end{equation}
The first of these choices appears also in
\cite[(8.4.3)]{LusChred}. On the right in both isomorphisms is the
additive group of elements of order prime to $p$ in ${\mathbb
Q}/{\mathbb Z}$. The isomorphisms exist essentially because all cyclic
groups of the same order are isomorphic; but they cannot be chosen
naturally. The field $\overline{\mathbb{Q}}_{\ell}$ appears in the second
because the methods of \'etale cohomology employed by Deligne and
Lusztig produce representations not on complex vector spaces but
rather on vector spaces over $\overline{\mathbb{Q}}_{\ell}$.

In this paper, we define the Weil--Deligne group for a finite field,
and use these to formulate the Langlands correspondence for the finite
field. Each fiber of this correspondence should be parametrized by
irreducible representations of a finite group attached to a Langlands
parameter, just as for other Langlands correspondences. 
In a subsequent paper by the first author, we construct the Langlands correspondence for finite
fields under the good prime assumption, and discuss a conjectural relation between the Langlands correspondence for finite fields and the categorical local Langlands correspondence. 

\subsection*{Acknowledgements}
The origin of the ideas in this paper was presented by the second author in his talk at the 
workshop ``New Developments in Representation Theory'' at IMS, Singapore in March 2016, which the first author attended. We thank the organizers of the workshop for providing that opportunity. 

\section{Langlands dual group}\label{sec:dual}
\setcounter{equation}{0}

The point of root data is that they provide a combinatorial way to
specify a reductive group. Here is a statement.

\begin{theorem} \label{thm:gpsareroots} Let $k$ be a field with separable closure
  $k^{\mathrm{sep}}$, and let $G$ and be a connected reductive algebraic group
  defined over $k$. Fix a Borel pair $(T,B)$ (not necessarily defined over $k$).  Write
  $\mathcal{BR}$ for the corresponding based root datum
  (\cite[Definition 2.10]{ABVLangclareal}).
  \begin{enumerate}
  \item There is a natural action of the Galois group $\Gamma =
  \Gal(k^{\mathrm{sep}}/k)$ on the based root datum. This action depends
  only on the inner class of the $k$-rational form $G$. The action factors
  through the Galois group of a finite Galois extension $E/k$.
  \item Suppose $T'\subset B' \subset G'$ is a Borel pair 
    in another connected reductive algebraic group over $k^{\mathrm{sep}}$, with based root datum $\mathcal{BR}'$, and that
    $$\Xi \colon \mathcal{BR} \rightarrow \mathcal{BR}'$$
    is an isomorphism of based root data, then $\Xi$ is induced by an
    isomorphism of algebraic groups over $k^{\mathrm{sep}}$
    $$\xi \colon (T\subset B \subset G) \rightarrow (T' \subset B'
    \subset G')$$
    The isomorphism $\xi$ is uniquely determined up to pre-composition
    with $\Ad(t)$ (for some $t\in T$), or post-composition with
    $\Ad(t')$  (for some $t'\in T'$).
    \item Suppose that we have pinnings ${\mathcal P}$ for $T\subset B
      \subset G$ and ${\mathcal P}'$ for $T'\subset B'\subset
      G'$. Then $\Xi$ is induced by a {\em unique} isomorphism of
      algebraic groups
      $$\xi_{{\mathcal P},{\mathcal P}'} \colon (G,{\mathcal P})
      \rightarrow (G',{\mathcal P}').$$
      \item Suppose $\mathcal{BR}''$ is any based root datum. Then
        there is a Borel pair $T'' \subset B''$, with a pinning
        ${\mathcal P}''$, in a
        reductive algebraic group $G''$ over $k^{\mathrm{sep}}$ with the
        property that the corresponding based root datum is $\mathcal{BR}''$. Because of (3), the pair $(G'',{\mathcal P}'')$ is
        unique up to a unique isomorphism.
        \item In the setting of (4), suppose in addition that
          $\mathcal{BR}''$ is endowed with an action of $\Gamma =
          \Gal(k^{\mathrm{sep}}/k)$ that factors through the Galois group
          of a finite Galois extension $E/k$. Then there is a
          unique definition of $G''$ over $k$ with the following properties:
          \begin{enumerate}
            \item The torus $T''$ is defined over $k$, and the
              corresponding action of $\Gamma$ on $X^*(T'')$ is the
              given one on $\mathcal{BR}''$.
              \item Each map $\phi_{\alpha''} \colon \SL_2 \to G''$ in the pinning
                ${\mathcal P}''$ is defined over $E$ (using the
                standard definition of $\SL_2$ over $E$).
          \end{enumerate}
          As a consequence of these properties, $B''$ is also defined
          over $k$, so that $G''$ is quasisplit.
  \end{enumerate}
\end{theorem}

\begin{corollary}\label{cor:distaut} Suppose $(G,{\mathcal P})$ is a
  connected reductive
  algebraic group with a pinning over a separably closed field
  $k^{\mathrm{sep}}$, and $\mathcal{BR}$ is the corresponding based root
  datum. Write $\Aut(G,{\mathcal P})$ for the group of algebraic
  automorphisms of $G$ preserving the pinning (in the weak sense of
  permuting the collection of maps from $\SL_2$ to $G$). Then there is
  a natural isomorphism
\begin{equation}\label{e:AutGP}
\Aut(G,{\mathcal P}) \simeq \Aut(\mathcal{BR})\colon
\end{equation}
  that is, every automorphism of the based root datum of $G$ arises
  from a unique algebraic group automorphism preserving the pinning.

  Every algebraic automorphism of $G$ differs by an inner automorphism
  from one preserving ${\mathcal P}$; and the only inner automorphism
  preserving ${\mathcal P}$ is the identity. Consequently there is a
  semidirect product decomposition
  $$\Aut(G) \simeq \Int(G) \rtimes \Aut(G,{\mathcal P}).$$
  \end{corollary}

\begin{definition}\label{def:Lgpred} Suppose $G$ is a reductive algebraic group defined
over the field $k$, and that $T\subset B \subset G$ is a Borel pair in $G_{k^{\mathrm{sep}}}$; we do not require that $T$ or $B$ be defined
over $k$.  Let $\Gamma= \Gal(k^{\mathrm{sep}}/k)$ act on the based root
datum $\mathcal{BR}$ as in Theorem \ref{thm:gpsareroots}(1). Suppose
$K = K^{\mathrm{sep}}$ is another field, assumed to be
separably closed. A {\em dual group to $G$ over $K$} is a pinned
reductive algebraic group $(G^{\vee},\mathcal{P}^{\vee})$ with
based root datum equal to the dual based root datum
$$\mathcal{BR}^\vee = (X_*,\Pi^\vee,X^*,\Pi).$$
(According to Theorem \ref{thm:gpsareroots}(4), the pinned group
$(G^{\vee},\mathcal{P}^{\vee})$ is unique up to a unique
isomorphism.)

We let $\Gamma$ act on the dual based root datum $\mathcal{BR}^{\vee}$ by the inverse transpose of its action on $\mathcal{BR}$. Because of the uniqueness of $(G^{\vee},\mathcal{P}^{\vee})$, $\Gamma$ acts on $(G^{\vee},\mathcal{P}^{\vee})$ by transport of structure.  (This is Corollary
\ref{cor:distaut}.) The {\em $L$-group of $G$ over $K$}
is the semidirect product
$${}^L G = \Gal(k^{\mathrm{sep}}/k) \ltimes G^{\vee}(K).$$
This action is by algebraic automorphisms, so ${}^L G$
is a pro-algebraic group over $K$. It is the inverse limit of the
algebraic groups
$$\Gal(E/k)\ltimes G^{\vee}(K)$$
with $E$ a finite Galois extension of $k$. The $L$-group of $G$ depends only
on the inner class of the $k$-rational form $G$.
\end{definition}

The general shape of a ``Langlands conjecture'' for group
representations is
\begin{equation}\label{eq:gsLangconj}
\text{\parbox{9cm}{irreducible representations of $G(k)$ on
  $K$-vector spaces (up to equivalence) should fall into disjoint
  {\em packets} $\Pi_\phi$ indexed by Langlands parameters $\phi$ (up to
      conjugation by $G^{\vee}(K)$).}}
\end{equation}
In this conjecture, a {\em Langlands parameter} is a group homomorphism
\begin{equation}
  \phi\colon W_k \rightarrow {}^L G
\end{equation}
subject to requirements including
\begin{enumerate}
\item $\im(\phi)$ is semisimple;
\item $\phi$ is compatible with the natural projections to 
$\Gal(k^{\mathrm{sep}}/k)$. 
\end{enumerate}
In this definition of Langlands parameter, the group $W_k$ is a {\em
  Weil group} for the field $k$. (Weil groups were defined for local
and global fields in \cite{WeiSurcorpcl}. We have not tried to determine
whether Weil's motivation in that paper, a good formulation of class
field theory, can be made to suggest anything about the case of finite
fields that we are now interested in.) A Weil group is required to be
equipped with a natural homomorphism
\begin{equation}\label{e:galmap}
  \pi_k \colon W_k \rightarrow \Gal(k^{\mathrm{sep}}/k)
\end{equation}
(so that condition (2) in the definition of Langlands parameter makes sense). 

Recall that for a local field $E$, a Weil group $W_E$ is a modified
Galois group. In particular, there is always a homomorphism
\begin{equation}
\pi_E\colon W_E \rightarrow \Gal(E^{\mathrm{sep}}/E),
\end{equation}
with dense image, whose kernel is an abelian subgroup of
$W_E$. 

\section{Weil groups of finite fields}\label{sec:Frob}
\setcounter{equation}{0}
	Suppose ${\mathbb F}_q$ is a finite field, and that $\overline{\mathbb
		F}_q$ is an algebraic closure. We know that
	\begin{equation*}
		\Gamma = \Gal(\overline{\mathbb F}_q/{\mathbb F}_q) = \varprojlim_m
		{\mathbb Z}/m{\mathbb Z};
	\end{equation*}
	the generator of this group is the {\em arithmetic Frobenius}
	\begin{equation*}
		\sigma_q \colon \overline{\mathbb F}_q \rightarrow \overline{\mathbb F}_q,
		\qquad \sigma_q (x) = x^q.
	\end{equation*}

For an algebraic variety $X$ over $\mathbb{F}_q$,
let $F \colon X \to X$ be the $q$-th power geometric Frobenius morphism.
We will try to
	write $\sigma_q$ for anything related to a Galois {\em group} action (so that
	$\sigma_q$ is invertible, and typically only ${\mathbb F}_q$-linear) and $F$
	for anything related to an $\overline{\mathbb F}_q$-morphism (so often
	{\em not} invertible).

The following definition is motivated by \cite[\S 3]{MacZetafingen}.

\begin{definition}\label{def:finiteweilZ}
	We put
	\begin{equation}\label{e:finiteweil}
		I_k = \varprojlim_m {\mathbb F}_{q^m}^\times
	\end{equation}
	where the limit is taken over the norm maps
	\begin{equation}
		N_{md,m}\colon {\mathbb F}_{q^{md}}^\times \rightarrow {\mathbb
			F}_{q^m}^\times .
	\end{equation}
	We define the {\em Weil group of $k$} by
	\begin{equation}
		W_k = I_k \rtimes \langle \sigma_q \rangle
	\end{equation}
	where the conjugation by $\sigma_q$ acts on $I_k$ as $q$-th power.
\end{definition}

\begin{remark}
	There is a natural system $\{ \mathbb{F}_{q^m}^{\times} \rtimes \Gal (\mathbb{F}_{q^m}/\mathbb{F}_q) \}_{m \geq 0}$ of finite quotients of $W_k$. These look similar to $W_{\mathbb{R}}= \mathbb{C}^{\times} \rtimes \Gal (\mathbb{C} /\mathbb{R})$.
\end{remark}

\begin{definition}\label{def:finiteRedweil}
Suppose $k = {\mathbb F}_q$ is a finite field. Suppose $G$ is a reductive algebraic group over
${\mathbb F}_q$, $K$ is an algebraically closed field, and ${}^L G$ is the $L$-group of $G$ over $K$
(Definition \ref{def:Lgpred}). A {\em Weil L-parameter} is a group homomorphism
		\begin{equation}
			\phi\colon W_k\rightarrow {}^L G
		\end{equation}
such that
		\begin{enumerate}
			\item the map $\phi$ is compatible with projections to $\Gal (\overline{k}/k)$;
			\item $\im(\phi)$ is semisimple; and
			\item $\phi|_{I_k}$ factors to some finite quotient ${\mathbb
				F}_{q^m}^\times$ of $I_k$.
		\end{enumerate}

For a Weil L-parameter $\phi$, write $\phi_0$ for $\phi|_{I_k}$.
The (pointwise) {\em stabilizer} of $\phi_0$ is
		\begin{equation}\label{e:stabphi}
			Z_{G^\vee} (\phi_0 ) = \{y \in G^\vee \mid
			\Ad(y)(\phi(w)) = \phi(w) \ \ (w\in I_k)\};
		\end{equation}
this is a (possibly disconnected) equal rank reductive subgroup of
$G^{\vee}$. The Dynkin diagram of $(G^{\vee})^\phi$ is obtained from
the extended Dynkin diagram of $G^{\vee}$ by deleting a nonempty set
of vertices in each simple factor.

We say that two Weil L-parameters $\phi$ and $\phi'$ are equivalent
if the following condition is satisfied: there is $g \in G^{\vee}$ such that
$\Ad (g)(\phi_0)=\phi_0'$, and the images of $\Ad (g)(\phi (\sigma_q))$ and
$\phi' (\sigma_q)$ in ${}^L G/Z_{G^\vee} (\phi'_0 )$ are same.

We write
		\begin{equation}\label{e:Langparam}
			\Phi_{{\mathbb F}_q}(G)
		\end{equation}
for the set of equivalence classes of Weil L-parameters.
\end{definition}

\begin{definition}
We say that
\[
 \phi_0 \colon I_k \to {}^L G
\]
is an {\em inertial L-parameter} if it is the restriction of a Weil L-parameter to $I_k$.
We say that two inertial L-parameters $\phi_0$ and $\phi_0'$ are {\em equivalent} if they
are conjugate by some $g \in G^{\vee}$.
\end{definition}

\begin{proposition}\label{prop:Fqtor} A torus $T$ defined over
		${\mathbb F}_q$ is the
		same thing as a lattice automorphism
		$$\sigma_q^*\colon X^*(T) \rightarrow X^*(T)$$
		of finite order; or equivalently the inverse transpose automorphism
		$$(\sigma_q)_*\colon X_*(T) \rightarrow X_*(T), \qquad (\sigma_q)_* =
		{}^t(\sigma_q^*)^{-1}.$$
		The algebraic functions on $T$ are $\overline{k}$-linear combinations
		of the rational characters in $X^*(T)$. The geometric Frobenius
		morphism (of algebraic groups defined over ${\mathbb F}_q$) $F\colon T
		\rightarrow T$ sends the character $\lambda$ to
		$q(\sigma_q^*)^{-1}(\lambda)$.
	\end{proposition}

	We need also

	\begin{proposition}\label{prop:fintor} Suppose $T$ is a torus over
		${\mathbb F}_q$ with
		geometric Frobenius map $F\colon T\rightarrow T$. Then
		\begin{enumerate}
			\item $T(\overline{\mathbb F}_q) = \displaystyle{\bigcup_{m\ge 1}} {\mathbb
				F}_{q^m}^\times \otimes_{\mathbb Z} X_*(T)$;
			\item $T(\overline{\mathbb F}_q) = \displaystyle{\bigcup_{m\ge 1}}
			T({\mathbb F}_{q^m})
			= \displaystyle{\bigcup_{m\ge 1}} T^{F^m}$.
			\item The norm map
			$$\begin{aligned}
				N_{md,m}\colon T({\mathbb F}_{q^{md}}) &\rightarrow T({\mathbb
					F}_{q^m}),\\
				N_{md,m}(t) &= t\cdot F^m(t)\cdot F^{2m}(t) \cdots
				F^{(d-1)m}(t)
			\end{aligned}$$
			is surjective, with kernel equal to $(F^m-1)T({\mathbb F}_{q^{md}})$.
		\end{enumerate}
	\end{proposition}
	\begin{proof} The first assertion is trivial. The second
		says that every point over the algebraic closure is defined over some
		finite extension. (The only issue requiring care is that the subgroups
		defined in the first assertion are {\em not} the points over ${\mathbb
			F}_{q^m}$ unless the torus is split.) The third assertion is a
		consequence of Proposition 3.2.2 of \cite{CarFinLie}.
	\end{proof}

	\begin{corollary}\label{cor:torchar}
		Suppose $T$ is a torus over ${\mathbb F}_q$ with
		geometric Frobenius map $F\colon T\rightarrow T$, and suppose $K$ is
		an algebraically closed field. Write
		$$\widehat{T({\mathbb F}_{q^m})} = \Hom(T({\mathbb
			F}_{q^m}),K^\times)$$
		for the group of one-dimensional characters over $K$. Write $F$ also
		for the automorphism of $\widehat{T({\mathbb F}_{q^m})}$ induced by
		the geometric Frobenius.  Then
		$$\begin{aligned}\widehat{T({\mathbb F}_q)} &\simeq \text{characters of
				$T({\mathbb F}_{q^m})$ factoring through $N_{m,1}$} \\
			&\simeq \widehat{T({\mathbb F}_{q^m})}^F.\end{aligned}$$
	\end{corollary}

	Corollary \ref{cor:torchar} will be the key to formulating a Langlands
	parametrization of torus characters.

	\begin{proposition}\label{prop:torparam} Suppose $T$ is a torus over
		${\mathbb F}_q$, $K$
		is an algebraically closed field, and ${}^L T$ is the $L$-group of $T$
		over $K$. Then there is a natural bijection
		$$\Phi_{{\mathbb F}_q}(T) \longleftrightarrow \widehat{T({\mathbb F}_q)},
		\qquad \phi \longleftrightarrow \xi(\phi)$$
		between the equivalence classes of Weil L-parameters of $T$ and characters of $T({\mathbb
			F}_q)$.
	\end{proposition}

	Here both the characters and the $L$-groups are defined using the
	algebraically closed field $K$. We do {\em not} make any assumption on
	the characteristic of $K$: both sides of the bijection may be smaller
	if the characteristic of the representing field $K$ divides the order
	of $T({\mathbb F}_q)$.

	\begin{proof}
Corollary \ref{cor:torchar} provides natural inclusions
			\begin{equation}
				\widehat{T({\mathbb F}_{q^m})} \subset \widehat{T({\mathbb
						F}_{q^{dm}})}
			\end{equation}
defined by norm maps; so we can consider the increasing ``union''
			\begin{equation}
				\widehat{T({\mathbb F}_{q^\infty})} = \varinjlim_m \widehat{T({\mathbb
						F}_{q^m})}.
			\end{equation}
Here ``increasing'' for positive integers is defined by divisibility.

We use the geometric Frobenius morphism $F$ of $T$ described in
Proposition \ref{prop:fintor}. Fix a positive integer $m_0$ so that
$F^{m_0}$ is equal to multiplication by $q^{m_0}$ (on the lattice
$X^*(T)$, for example). (We can take for $m_0$ the order of the
arithmetic Frobenius automorphism $\sigma_q^*$ of Proposition
\ref{prop:Fqtor}.) According to Proposition \ref{prop:fintor}, for any
$m$ divisible by $m_0$,
			\begin{equation}
				T({\mathbb F}_{q^m}) = \Hom_{\mathbb Z}(X^*(T),{\mathbb F}^\times_{q^m}),
			\end{equation}
so
			\begin{equation}\label{e:torchar}
				\begin{aligned}
					\widehat{T({\mathbb F}_{q^m})} &= \Hom(T({\mathbb
						F}_{q^m}),K^\times)\\
					&=\Hom(\Hom_{\mathbb Z}(X^*(T),{\mathbb F}^\times_{q^m}),K^\times)\\
					&=\Hom({\mathbb F}^\times_{q^m},X^*(T)\otimes_{\mathbb Z} K^\times)\\
					&=\Hom({\mathbb F}^\times_{q^m},T^{\vee}(K))
				\end{aligned} \quad(\text{$m$ divisible by $m_0$}).\end{equation}
By \eqref{e:torchar}, we conclude
			\begin{equation}\label{e:torinftychar}
				\widehat{T({\mathbb F}_{q^\infty})} = \Hom(I_k,T^{\vee}(K)),
			\end{equation}
the continuous homomorphisms that factor to some quotient
$\mathbb{F}_{q^m}$ of $I_k$. The arithmetic Frobenius $f$ acts on
these homomorphisms by acting on the domain by $\sigma_q^{-1}$ and on the
range by $\sigma_q^*$. (To see that this is the correct action, one
can follow the Galois action of the arithmetic Frobenius $\sigma_q$ through the
isomorphisms of \eqref{e:torchar}). The fixed points of $\sigma_q^m$
under this action are the characters of $T({\mathbb F}_{q^m})$:
			\begin{equation}\label{e:torparam}
				\widehat{T({\mathbb F}_{q^m})} = \Hom(I_k,T^{\vee}(K))^{\sigma_q^m}
				\qquad (m \ge 1).
			\end{equation}
Here there is no divisibility requirement on $m$. We considered first
very divisible $m$ to get a simple computation of characters; but Corollary
\ref{cor:torchar} then gives a result for {\em all} $m \ge 1$.)

On the other hand, the right side of \eqref{e:torparam} is exactly
			\begin{equation}\label{e:torLang}
				\Hom(I_k,T^{\vee}(K))^{\sigma_q^m} = \Phi_{{\mathbb F}_{q^m}}(T)
				\qquad (m \ge 1). 
			\end{equation}
	\end{proof}

We are more or less ready to state a Langlands classification for
finite groups of Lie type based on \cite{DeLuRep}. In order to make it
more explicit, we need one more definition.

\begin{definition}\label{def:finRedrigid} In the setting of Definition
  \ref{def:finiteRedweil}, a {\em rigid Weil L-parameter} is a
  pair $(\phi,T^{\vee})$ such that
  \begin{enumerate}
    \item $\phi\colon W_k \rightarrow {}^L G$ is a
      Weil L-parameter;
    \item $T^{\vee}$ is a maximal torus in $G^{\vee}$ such that $\phi
      (I_k) \subset T^{\vee}$ and $\im \phi \subset N_{{}^L G}(T^{\vee})$.
  \end{enumerate}
We say that two rigid Weil L-parameters $(\phi,T^{\vee})$ and
$(\phi',T'^{\vee})$ are {\em equivalent} if the following condition is satisfied:
there is $g \in G^{\vee}$ such that $\Ad (g)(\phi_0)=\phi_0'$, $\Ad
(g)(T^{\vee})=T'^{\vee}$, and the images of $\Ad (g)(\phi (\sigma_q))$ and
$\phi' (\sigma_q)$ in ${}^L G/T'^{\vee}$ are same.
\end{definition}

\begin{proposition} Suppose we are in the setting of Definitions
  \ref{def:finiteRedweil} and \ref{def:finRedrigid}.
  \begin{enumerate}
  \item Any Weil L-parameter $\phi$ is equivalent to the first term of a rigid Weil L-parameter $(\phi',T^{\vee})$.
    \item Assume that two rigid Weil L-parameters $(\phi,T^{\vee})$ and $(\phi',T^{\vee})$ 
      satisfy $\phi|_{I_k}=\phi'|_{I_k}$. Then there is $g \in N_{Z_{G^{\vee}}(\phi_0)}
     (T^{\vee})$ such that $\phi'(\sigma_q)=\phi (\sigma_q)g$. They
      are equivalent if and only if $g \in N_{Z_{G^{\vee}}(\phi_0)_0} (T^{\vee})$.
  \end{enumerate}
\end{proposition}
\begin{proof}
For the first, the subgroup
  $\phi(I_k)\subset G^{\vee}$ is cyclic and
  semisimple, and therefore contained in a maximal torus $T^{\vee}$
  of $G^{\vee}$; then automatically $T^{\vee} \subset Z_{G^{\vee}}(\phi_0)$. Then
  $$\Ad(\phi(\sigma_q)^{-1})(T^{\vee}) = T_0^{\vee}$$
  is another torus of $Z_{G^{\vee}}(\phi_0)$; so there is an
  element $g_0 \in Z_{G^{\vee}}(\phi_0)_0$ so that
  $$\Ad(g_0)(T^{\vee}) = T_0^{\vee}.$$ 
If we define $g=\phi(\sigma_q) g_0$, then the conclusion is that
$$\Ad(g)(T^{\vee}) = T^{\vee},$$ 
so that $g\in N_{{}^L G}(T^{\vee})$. Because $g_0 \in G^{\vee}$, $g$
has the same image $\sigma_q \in \Gal(\overline{k}/k)$ as $\phi(\sigma_q)$. Similarly,
because $g_0$ centralizes the image of $\phi_0$, we have
$$\Ad(g)(\phi(w)) = \Ad(\phi(\sigma_q))(\phi(w)) = \phi(w)^q$$ for $w \in I_k$. 
We define $\phi'$ by $\phi'|_{I_k}=\phi|_{I_k}$ and 
$\phi' (\sigma_q) =g$. 
Then $(\phi',T^{\vee})$ is a rigid Weil L-parameter, proving
(1).

The second assertion follows easily from the definition. 
\end{proof}

Because it is so central to this paper, we essentially repeat the
proof of the proposition by explaining how to list all rigid Langlands
parameters. Write
  \begin{equation}
    T^{\vee}_0 \subset B_0^{\vee} 
  \end{equation}
  for the Borel pair specified by
  the pinning in the definition of $G^{\vee}$ (Definition
  \ref{def:Lgpred}). Since equivalence of rigid parameters is
  conjugation by $G^{\vee}$, every rigid parameter has a
  representative with torus part $T^{\vee}_0$; so we seek to
  enumerate these. Now
  \begin{equation}\label{e:extW}\begin{aligned}
    N_{{}^L G}(T^{\vee}_0) &= \Gal(\overline k/k) \ltimes N_{{}^\vee
      G}(T^{\vee}_0), \\
    N_{{}^L G}(T^{\vee}_0)/T^{\vee}_0 &=
    \Gal(\overline k/k) \ltimes W(G^{\vee},T^{\vee}_0).
    \end{aligned}
  \end{equation}
  We write
  \begin{equation}
    W = W(G^{\vee},T^{\vee}_0). 
  \end{equation}

It is now clear that equivalence classes of rigid Weil L-parameters
are exactly the same thing as $W$-orbits of
pairs
\begin{equation}\label{frobeq}
  (\phi_0,x_0), \qquad \phi_0\colon I_k \rightarrow T^{\vee}_0, \quad
  x_0\in \sigma_q \cdot W \subset \Gal(\overline k/k)\ltimes W
\end{equation}
subject to the requirement that
\begin{equation}
  \Ad(x)(\phi_0(w)) = \phi_0(w^q),
  \end{equation}
\begin{equation}\label{e:frobeqm}
  \Ad(x^m)(\phi_0(w)) = \phi_0(w^{q^m}). 
  \end{equation}
If $m_0$ is a positive integer divisible by the order of every element
of $\sigma_q W$, then \eqref{e:frobeqm} implies that the image of $\phi_0$
consists of elements of order dividing $q^{m_0} - 1$. This means in
particular that $\phi_0$ must factor to ${\mathbb F}^\times_{q^{m_0}}$. If we
choose a multiplicative generator $\eta$ of this group, then $\phi_0$
is determined by the single element $y_0 = \phi_0(\eta)\in T^{\vee}$,
which is required only to satisfy
\begin{equation}\label{e:frobeq2}
  x_0y_0x_0^{-1} = y_0^q.
\end{equation}
That is, {\em for each element $x_0\in \sigma_q \cdot W$, the set of rigid Langlands
parameters $(\phi_0,x_0)$ may be identified with the finite subgroup of
elements $y_0\in T^{\vee}$ (necessarily of order dividing $q^m-1$)
satisfying \eqref{e:frobeq2}}.

Therefore the equivalence classes of rigid Weil L-parameters are
partitioned by $W$-conjugacy classes in the coset $\sigma_q \cdot W$; and if a
conjugacy class has representative $x_0$, then the corresponding set
of parameters may be labelled (not canonically) by the finite group
defined by \eqref{e:frobeq2}.

With these explicit descriptions of L-parameters in hand, we
can relate them to our finite Chevalley group $G(k)$.

\begin{proposition}\label{prop:redtorparams}
  Suppose $G$ is a connected reductive algebraic group defined over
  the finite field $k={\mathbb F}_q$, $B_0\subset G$ is a Borel
  subgroup defined over $k$, and $T_0\subset B_0$ is a maximal torus
  defined over $k$. Let $\Gal(\overline k/k)$ act on $W(G,T_0)$ as in
  Theorem \ref{thm:gpsareroots}, and form the semidirect product
  $$ \Gal(\overline k/k) \ltimes W(G,T_0).$$
  Fix also a second algebraically closed field $K$, over which we
  define $L$-groups and represent $k$-groups.
  \begin{enumerate}
    \item\label{en:semidiso} 
    The semidirect product above is naturally isomorphic to
      \eqref{e:extW}.
      \item\label{en:ratconjmax}
       The $G(k)$-conjugacy classes of maximal tori in $G$
        defined over $k$ are naturally indexed (by the Frobenius
        action) by $W$-conjugacy classes in $\sigma_q \cdot W$. 
        This bijection is given by sending $\Ad (g) T_0$ to the image of $g^{-1} F(g) \in N_{G}(T_0)$ in $W \cong \sigma_q \cdot W$. 
        \item\label{en:isoWeylstab} 
        If $x_0\in \sigma_q \cdot W$ and $T_{x_0}$ is a corresponding
          maximal torus defined over $k$, then there is a natural
          isomorphism
          $$N_{G}(T_{x_0})(k)/T_{x_0}(k) \simeq W(G^{\vee},T^{\vee})^{x_0}.$$
        \item\label{en:bijTrig}
         In the setting of \ref{en:isoWeylstab}, there is a natural
          bijection
   \begin{align*}
   &\widehat{T_{x_0}(k)}/(N_{G}(T_{x_0})(k)/T_{x_0}(k)) 
   \\ 
   &\simeq
         \{ \text{rigid Weil L-parameters $(\phi,\widehat{T})$ such that $\phi(\sigma_q)$ is a lift of $x_0$} \}/{\sim} . 
             \end{align*} 
\item\label{en:geomconjineL} 
In the bijection of \ref{en:bijTrig}, suppose two characters $\theta_1\in
  \widehat{T_{x_1}(k)}$ and $\theta_2\in
  \widehat{T_{x_2}(k)}$ correspond to the rigid parameters
  $(\phi_i,x_i)$. Then the pair are geometrically conjugate
  (\cite[Definition 5.5]{DeLuRep}) if and only if the inertial L-parameters $\phi_1|_{I_k}$ and $\phi_2|_{I_k}$ are equivalent.
  \end{enumerate}
\end{proposition}
\begin{proof} The claim \ref{en:semidiso} follows from the construction of the
  $L$-group in Definition \ref{def:Lgpred}. The claim \ref{en:ratconjmax} is
  \cite[Corollary 1.14]{DeLuRep}. 
As for \ref{en:isoWeylstab}, we have 
$$N_{G}(T_{x_0})(k)/T_{x_0}(k) \simeq W(G,T_{x_0})^F \simeq W(G,T_0)^{x_0} \simeq W(G^{\vee},T^{\vee})^{x_0}$$
by \cite[Proposition 4.4.1]{DiMiRepLie2nd}. 
  The claim \ref{en:bijTrig} is Proposition
  \ref{prop:torparam}. For the claim \ref{en:geomconjineL}, suppose $m\ge 1$; consider the
  (surjective) norm homomorphisms
  $$N\colon T_{x_i}({\mathbb F}_{q^m}) \rightarrow T_{x_i}({\mathbb F}_q)$$ of
  Proposition \ref{prop:fintor}. In the bijections of (3), the pairs
  $(T_i({\mathbb F}_{q^m}),\theta_i\circ N)$ clearly correspond to the
  rigid Weil L-parameters $(\phi_i|_{W_{\mathbb{F}_{q^m}}},\widehat{T})$. We choose $m$
  so that $\Ad(x_i^m)$ is
  trivial on $T^{\vee}$; so equivalence of the rigid Weil L-parameters is
  the same as equivalence of $\phi_1|_{I_k}$ and $\phi_2|_{I_k}$ by \ref{en:bijTrig}. 
 \end{proof}

The proposition says that equivalence classes of rigid Weil L-parameters are in one-to-one correspondence with $G(k)$-conjugacy
classes of pairs $(T,\theta)$, with $T$ a maximal torus in $G$ defined
over $k$, and $\theta \in \widehat{T(k)}$. A version of this is in
\cite[(5.21.5)]{DeLuRep}.

The main results of \cite{DeLuRep} concern the case
\begin{equation}
  K = \overline{\mathbb Q}_\ell,
\end{equation}
with $\ell$ any prime not equal to $p$. In that setting, Deligne and
Lusztig define a virtual $K$-representation $R_{T_{x_0}}(\theta)$ of
$G(k)$ for every $(T_{x_0},\theta)$ as in the proposition.

Here is a way to write the Deligne--Lusztig results as a Langlands
classification for finite groups of Lie type.

\begin{theorem}[Deligne--Lusztig \cite{DeLuRep}]\label{thm:DL} Suppose $G$
  is a connected reductive algebraic group defined over the finite
  field $k={\mathbb F}_q$. Consider
  $$ K = \overline{\mathbb Q}_\ell,$$
  with $\ell$ any prime not equal to $p$. Suppose $(\phi,T^{\vee})$
  is a rigid Weil L-parameter (Definition
  \ref{def:finRedrigid}). Let $(T,\theta)$ be a corresponding pair
  consisting of a maximal torus defined over $k$ and a
  $K^\times$-valued character of $T(k)$ (Proposition
  \ref{prop:redtorparams}(3)). Define
  $$R_T(\theta) = \text{virtual $K$-representation of $G(k)$}$$
  as in \cite{DeLuRep}.
  \begin{enumerate}
  \item The virtual representations $R_{T_1}(\theta_1)$ and
    $R_{T_2}(\theta_2)$ have irreducible summands in common
    only if $(T_1,\theta_1)$ and $(T_2,\theta_2)$ are geometrically
    conjugate; that is, only if the corresponding rigid Weil L-parameters have equivalent underlying inertial L-parameters.
  \item Every irreducible $G(k)$ representation over $K$ appears as an
    irreducible summand of some $R_T(\theta)$.
  \end{enumerate}
\end{theorem}

\begin{definition}\label{def:Lpkt} In the setting of Theorem 5.9 (so
  that $K=\overline{\mathbb Q}_\ell$) write $\Pi(G(k))$ for the set of
  irreducible $K$-representations of $G(k)$.

  Suppose $\phi_0$ is an inertial L-parameter. The {\em $L$-packet of
    $\phi_0$} is
  $$\Pi_{\phi_0}(G(k)) = \{\pi \in \Pi(G(k)) \mid \text{$\pi$ appears in
  $R_T(\theta)$}\}$$
 for some character $\theta$ of some rational torus
 corresponding to a rigid Weil L-parameter $(\phi,T^{\vee})$ such that $\phi|_{I_k}$ is equivalent to $\phi_0$. 
\end{definition}

According to Theorem \ref{thm:DL}, the $L$-packets
partition $\Pi(G(k))$. 

What we want next is a more explicit description of the packets
$\Pi_{\phi_0}$. To begin, we ask how large these packets are.
\begin{theorem}[Deligne--Lusztig {\cite[Theorem 6.8]{DeLuRep}}]\label{thm:DLint}
In the setting of Theorem \ref{thm:DL}, write the decomposition of the
virtual representation $R_T(\theta)$ into irreducible representations
as
$$R_T(\theta) = \sum_{\pi \in \Pi_{\phi_0}(G(k))} m(\pi)\pi,$$
with each multiplicity $m(\pi)$ an integer. Then
$$\sum_{\pi\in \Pi_{\phi_0}(G(k))} m(\pi)^2 = |(W(G,T)^F)^{\theta}| =
|W((G^{\vee})^{\phi_0},T^{\vee})^x|.$$
Here the two Weyl groups appearing are identified by Proposition
\ref{prop:redtorparams}(3).
\end{theorem}

\begin{corollary}\label{cor:DLirr} In the setting of Theorem
  \ref{thm:DL}, suppose that the stabilizer $(G^{\vee})^{\phi_0}$ of the
  inertial L-parameter $\phi_0$ is a maximal torus in $G^{\vee}$. Then $\Pi_{\phi_0}(G(k))$ is a single irreducible representation,
  namely $\pm R_T(\theta)$.
\end{corollary}
\begin{proof}
The hypothesis is equivalent to the triviality of $W((G^{\vee})^{\phi_0},T^{\vee})$.
\end{proof}

Theorem \ref{thm:DLint} and Corollary \ref{cor:DLirr} suggest that the
size of the $L$-packet $\Pi_{\phi_0}(G(k))$ is controlled by the failure
of the (possibly disconnected) reductive group $(G^{\vee})^{\phi_0}$ to be
a torus. The first step is to enlarge the Weil group to the
Weil--Deligne group, which we introduce in the next section.

\section{Weil--Deligne groups and a Langlands correspondence for finite fields}\label{sec:WD} \setcounter{equation}{0}
We want to refine the partition of the representations $\Pi(G(k))$ in
\eqref{eq:gsLangconj}. We will turn next to a rough outline of the
idea introduced by Deligne for doing that.

The general shape of a Deligne's modified version of ``Langlands conjecture'' for group
representations is 
\begin{equation}\label{e:DelLangA}
\text{\parbox{9cm}{irreducible
      representations of $G(k)$ on $K$-vector spaces (up to
      equivalence) should fall into disjoint {\em packets}
      $\Pi_{\varphi}$ indexed by L-parameters
  $\varphi$ of Weil--Deligne type (up to conjugation by $G^{\vee}(K)$).}}
\end{equation}
The conjecture includes an idea about the structure of an
L-packet:
\begin{equation}\label{e:DelLangB}
\text{\parbox{9cm}{representations
  in a packet $\Pi_{\varphi}$ should be indexed approximately by some irreducible
  $G^{\vee}(K)$-equivariant local systems of $K$-vector spaces on $G^{\vee}(K)\cdot
  \varphi$; that is, by irreducible $K$-representations of the group of
  connected components $G^{\vee}(K)^{\varphi}/G^{\vee}(K)^{\varphi}_0$.}}
\end{equation}
An even more optimistic version (proven for real groups in \cite{ABVLangclareal})
is
\begin{equation}\label{e:DelLangC}
\text{\parbox{9cm}{the category of
      $G(k)$ representations built from a packet $\Pi_{\varphi}$ should
      be in duality with the category of $G^{\vee}(K)$-equivariant
      perverse sheaves on $G^{\vee}(K)\cdot\varphi$.}}
\end{equation}

In this conjecture, an {\em L-parameter of Weil--Deligne type} is an
algebraic group homomorphism 
\begin{equation}
  \varphi \colon \mathit{WD}_k \rightarrow {}^L G
\end{equation}
subject to requirements including
\begin{enumerate}
  \item $\im(\varphi|_{W_k})$ is semisimple;
\item $\varphi$ is compatible with the natural projections to 
  $\Gal(k^{\mathrm{sep}}/k)$. 

\end{enumerate}
In this definition of L-parameter of Weil--Deligne type, the group $\mathit{WD}_k$ is
a {\em Weil--Deligne group} for the field $k$.  We will not state
general requirements for a Weil--Deligne group, 
because we do not know how to
formulate them in a way consistent with our fond hope: that there is a
definition of something like an archimedean Weil--Deligne ``group'' (or
at least of its representations) that incorporates the modified notion
of ``Langlands parameter'' introduced in \cite{ABVLangclareal}. The same fond
hope asks also that this archimedean definition be made consistent with
Deligne's non-archimedean definition in \cite[8.3.6]{DelconstL}.

Here at any rate is a definition of a Weil--Deligne group for a finite field.  Each Weil L-parameter $\phi$ will have
several extensions to an L-parameter $\varphi$; we will
try to arrange a corresponding partition of each $L$-packet
$\Pi_\phi$ into several smaller packets.

\begin{definition}\label{def:finweilDel}
	The {\em Weil--Deligne group of the finite
		field $k$} is the locally pro-algebraic group
	scheme
		\begin{equation}
			\mathit{WD}_k =\Ga \rtimes W_k
		\end{equation}
over $\mathbb{Z}[1/p]$, where $(\sigma_q^n,w) \in W_k$ acts on
		$\Ga$ by the multiplication by $q^n$.

	Assume that the characteristic of $K$ is not $p$.
	An {\em L-parameter of Weil--Deligne type} is a locally
        pro-algebraic group homomorphism
		\begin{equation}
			\varphi \colon \mathit{WD}_k(K)\rightarrow {}^L G
		\end{equation}
		satisfying the following conditions:
		\begin{enumerate}
			\item The map $\varphi$ is compatible with projections to $\Gal (\overline{k}/k)$.
			\item $\varphi|_{I_k}$ factors to some finite quotient ${\mathbb
				F}_{q^m}^\times$ of $I_k$.
		\end{enumerate}
		We say that the L-parameter $\varphi$ is {\em special}
		if $\varphi|_{\Ga (K)}(1)$ is a special unipotent element of
		$(G^{\vee})^{\varphi (I_k)}$.
		We say that $\varphi$ is Frobenius semisimple if
		$\varphi(\sigma_q)$ is semisimple in ${}^L G$.
\end{definition}

\begin{lemma}\label{lem:pi0modZ}
Let $g=su \in G^{\vee}$ be the Jordan decomposition.
Then we have
\[
\pi_0 \left( Z_{Z_{G^{\vee}}(s)^{\circ}}(u)/
Z\left( Z_{G^{\vee}}(s)^{\circ} \right) \right)
\cong \pi_0 \left( Z_{Z_{G^{\vee}}(s)^{\circ}}(u)/Z(G^{\vee}) \right). 
\]
\end{lemma}
\begin{proof}
We take a Borel pair $T \subset B$ in $G_{k^{\sep}}$. We may assume that $s \in T^{\vee}$. Let $\Delta$ be the set of simple root of $G_{k^{\sep}}$ with respect to $T \subset B$. 
	Let $I \subset \Delta$ be the subset consisting of elements which are the simple coroots of $Z_{G^{\vee}}(s)^{\circ}$.
	Since $\mathbb{Z}\Delta/\mathbb{Z}I$ has no torsion, the natural map
	\[
	(X^*(T)/\mathbb{Z} I)_{\mathrm{tor}} \to
	(X^*(T)/\mathbb{Z} \Delta)_{\mathrm{tor}}
	\]is injective.
	Then $\pi_0 (Z(G^{\vee})) \to \pi_0 (Z(Z_{G^{\vee}}(s)^{\circ}))$ is surjective,
	since it is identified with
	\[
	(X^*(T)/\mathbb{Z} \Delta)_{\mathrm{tor}}^{\vee} \to
	(X^*(T)/\mathbb{Z} I)_{\mathrm{tor}}^{\vee}.
	\]
	Hence, we have
	\begin{align*}
		\pi_0 (Z_{Z_{G^{\vee}}(s)^{\circ}}(u)/Z(Z_{G^{\vee}}(s)^{\circ}))
		&\cong \pi_0 (Z_{Z_{G^{\vee}}(s)^{\circ}}(u))/\pi_0 (Z(Z_{G^{\vee}}(s)^{\circ}))\\
		&\cong \pi_0 (Z_{Z_{G^{\vee}}(s)^{\circ}}(u))/\pi_0 (Z(G^{\vee})) \\
		&\cong \pi_0 (Z_{Z_{G^{\vee}}(s)^{\circ}}(u)/Z(G^{\vee})),
	\end{align*}
	where we use the surjectivity of
	$\pi_0 (Z(G^{\vee})) \to \pi_0 (Z(Z_{G^{\vee}}(s)^{\circ}))$ at the second equality.
\end{proof}

We put
\[
 A_{Z_{G^{\vee}}(\varphi(I_k))^{\circ}}(\varphi(\Ga)) =
 \pi_0 (Z_{Z_{G^{\vee}}(\varphi(I_k))^{\circ}}(\varphi(\Ga))/Z(Z_{G^{\vee}}(\varphi(I_k))^{\circ})). 
\]
We define {\em Lusztig's canonical quotient $\overline{A}_{Z_{G^{\vee}}(\varphi(I_k))^{\circ}}(\varphi(\Ga))$
	of $A_{Z_{G^{\vee}}(\varphi(I_k))^{\circ}}(\varphi(\Ga))$} as in \cite[13.1]{LusChred} using the isomorphism given by Lemma \ref{lem:pi0modZ}.
We put $\varphi_0=\varphi|_{\Ga \times I_k}$ and
\begin{equation}
	A(\varphi_0)=\pi_0 \left( Z_{G^{\vee}}(\varphi_0)/Z(G^{\vee}) \right).
\end{equation}
 Further we put
\begin{equation}
	\overline{A}(\varphi_0)=A(\varphi_0)/\Ker ( A_{Z_{G^{\vee}}(\varphi(I_k))^{\circ}}(\varphi(\Ga)) \to \overline{A}_{Z_{G^{\vee}}(\varphi(I_k))^{\circ}}(\varphi(\Ga))).
\end{equation}
To obtain an enlargement of $\overline{A}(\varphi_0)$, we put
\begin{equation*}
	\widetilde{Z}(\varphi_0)=\{ (g,\sigma_q^m) \in {}^{L}G \mid
	\Ad ((g,\sigma_q^m))(\varphi_0 (x))=\varphi_0 (\Ad (\sigma_q^m)(x)) \ \textrm{for all $x \in \Ga \times I_k$}
	\} 
\end{equation*}
and 
\begin{equation}
	\widetilde{A}(\varphi_0)=\widetilde{Z}(\varphi_0) /
	\Ker (Z_{G^{\vee}}(\varphi_0) \to \overline{A}(\varphi_0)).
\end{equation}

We have 
$\varphi (\sigma_q) \in \widetilde{Z}(\varphi_0)$. 
Let $\overline{\varphi (\sigma_q)}$ be the image of 
$\varphi (\sigma_q)$ under the natural projection 
\[
\widetilde{Z}(\varphi_0) \to 
\widetilde{A}(\varphi_0). 
\]
We say that two 
L-parameters $\varphi$ and $\varphi'$ of Weil--Deligne type are equivalent 
if the following condition is satisfied: 
there is $g \in G^{\vee}$ such that 
$\Ad (g)(\varphi_0)=\varphi_0'$ and 
$\overline{\varphi (\sigma_q)}$ corresponds to 
$\overline{\varphi' (\sigma_q)}$ under the bijection 
\[
\widetilde{A}(\varphi_0) \cong 
\widetilde{A}(\varphi'_0) 
\]
induced by $\Ad (g)$, where 
$\varphi_0=\varphi|_{\Ga \times I_k}$ and 
$\varphi_0'=\varphi'|_{\Ga \times I_k}$. 
Let $\Phi_K(G)$ be the equivalence classes of 
Frobenius semisimple L-parameters over $K$ of $G$. 
We write $\Phi_K(G)_{\mathrm{sp}} \subset \Phi_K(G)$ for 
the equivalence classes of special ones. 

We put
\begin{equation}
	A_{\varphi} = Z_{\overline{A}(\varphi_0)} (\overline{\varphi (\sigma_q)}) . 
\end{equation} 
The following is a formulation of the Langlands correspondence for finite fields. 

\begin{conjecture}
There is a natural map 
\begin{equation}\label{eq:LLCfin}
\mathcal{L}_G \colon \Irr_{\overline{\mathbb{Q}}_{\ell}}(G(k)) \to 
\Phi_{\overline{\mathbb{Q}}_{\ell}}(G)_{\mathrm{sp}} 
\end{equation}
such that, for $\varphi \in \Phi_{\overline{\mathbb{Q}}_{\ell}}(G)_{\mathrm{sp}}$, 
we have a bijection between 
$\mathcal{L}_G^{-1}(\varphi)$ and 
$\Irr_{\overline{\mathbb{Q}}_{\ell}}(A_{\varphi})$. 
\end{conjecture}

\noindent
Naoki Imai\\
Graduate School of Mathematical Sciences, The University of Tokyo, 
3-8-1 Komaba, Meguro-ku, Tokyo, 153-8914, Japan \\
naoki@ms.u-tokyo.ac.jp\\

\noindent
David A. Vogan, Jr.\\
2-355, Department of Mathematics MIT, Cambridge, MA 02139, USA \\ 
dav@math.mit.edu

\end{document}